\documentclass[a4paper,12pt]{article}

\usepackage{hyperref}
\usepackage {amsfonts, amssymb, amscd,amsthm, amsmath, eucal}

\usepackage{amsbsy}
\usepackage[all]{xy}

\theoremstyle{plain} { \swapnumbers
  \newtheorem{thm}{Theorem}[subsection]
  
  \newtheorem{cor}[thm]{Corollary}
  \newtheorem{lem}[thm]{Lemma}
  \newtheorem{prop}[thm]{Proposition}

}

\theoremstyle{definition} { \swapnumbers
  \newtheorem{sssection}[thm]{}
}
\renewcommand{\subsubsection}{\sssection}

\newcommand{\Aff}{\mathbf {A}}
\newcommand{\Pro}{\mathbf {P}}

\newcommand \xra {\xrightarrow }

\newcommand \hra {\hookrightarrow }

\begin{document}

\title
{Rationally trivial quadratic spaces are locally trivial:III}

\author{I.Panin\footnote{The research of the first author
is partially supported by RFBR-grant 10-01-00551-a},
K.Pimenov \footnote{The research of second author was partially supported by RFBR grant  12-01-33057
"Motivic  homotopic cohomology theories on algebraic varieties"}}

\date{25.12.2012}

\maketitle

\begin{abstract}
Main results of \cite{Pa}, \cite{PaPi} are extended to the case of characteristic two.
The proof given in the present preprint is "elementary" and is characteristic free.

More precisely, let $R$ be a regular semi-local domain containing a field such that
all the residue fields are infinite.
Let $K$ be the fraction field of $R$.
Let
$(R^n,q: R^n \to R)$
be a quadratic space over $R$ such that the quadric $\{q=0\}$ is smooth over $R$.
If the quadratic space $(R^n,q: R^n \to R)$ over $R$ is isotropic over $K$, then there is a unimodular
vector $v \in R^n$ such that $q(v)=0$.

If $char(R)=2$, then in the case of even $n$ our assumption on $q$ is equivalent to the one that $q$ is a non-singular space in the sense of
\cite{Kn} and in the case of odd $n > 2$ our assumption on $q$ is equivalent to the one that $q$ is a semi-regular in the sense of
\cite{Kn}.

\end{abstract}

\section{Introduction}
Let $k$ be an infinite field,possibly $char(k)=2$, and let $X$ be a $k$-smooth irreducible affine scheme,let
$x_1, x_2,\dots, x_s \in X$
be closed points. Let $P$ be a free $k[X]$-module of rank $n>0$.
If $n$ is odd, then let $(P, q: P \to k[X])$
be a semi-regular quadratic module over $k[X]$ in the sense of
\cite[Ch.IV, \S3]{Kn}.
If $n$ is even, then let $(P, q: P \to k[X])$
be a non-singular quadratic space in the sense of
\cite[Ch.I, (5.3.5))]{Kn}.
(In both cases it is equivalent of saying that the $X$-scheme $Q:=\{q=0\} \subset \Pro^{n-1}_X$ is smooth over $X$).

Let $p: Q \to X$ be the projection. For a nonzero element
$g \in k[X]$ let $Q_g=p^{-1}(X_g)$.
Let $U=\text{Spec}(\mathcal O_{X,\{ x_1, x_2,\dots, x_s \}})$.
Set $_UQ=U \times _X Q$. For a $k$-scheme $D$ equipped with
$k$-morphisms
$U \leftarrow D$
and
$D \to X_g$
set
$_DQ= \ _UQ \times_U D$
and
$Q_{D,g}=D \times_{X_g} Q_g$.

\begin{prop}
\label{Main}
If $n>1$, then there exists a finite surjective \'{e}tale $k$-morphism
$U \leftarrow D$
of odd degree, a morphism
$D \to X_g$
and an isomorphism of the $D$-schemes
$_DQ \xleftarrow{\bar \Phi} Q_{D,g}$.
\end{prop}

Given this Proposition we may prove the following Theorem
\begin{thm}[Main]
Assume that $g \in k[X]$
is a non-zero element such that there is a section
$s: X_g \to Q$ of the projection
$Q_g \to X_g$.
Then there is a section
$s_U: U \to \ _UQ$
of the projection
$_UQ \to U$.
\end{thm}

\begin{proof}[Proof of Main Theorem]
We will give a proof of the Theorem only in the local case and left to the reader the semi-local case.
So, $s=1$ and we will write $x$ for $x_1$ and $\mathcal O_{X,x}$ for $\mathcal O_{X,\{ x_1 \}}$.
If $g \in k[X]-m_{x}$, then there is nothing to prove. Now let $g \in m_x$ then by
Proposition \ref{Main}
there is a a finite surjective \'{e}tale $k$-morphism
$U \leftarrow D$
of odd degree, a morphism
$D \to X_g$
and an isomorphism of the $D$-schemes
$_DQ \xleftarrow{\bar \Phi} Q_{D,g}$.

The section $s$ defines a section
$s_D=(id,s): D \to Q_{D,g}$ of the projection
$Q_{D,g} \to D$.
Further
$\bar \Phi \circ s_D: D \to \ _DQ$
is a section of the projection
$_DQ \to D$.
Finally,
if $p_1: \ _DQ \to \ _UQ$ is the projection, then
$p_1 \circ \bar \Phi \circ s_D : D \to \ _UQ$
is a $U$-morphism of $U$-schemes.
Recall that
$U \leftarrow D$
is a a finite surjective \'{e}tale $k$-morphism
of odd degree and $U$ is local with an infinite residue field.
Whence by a variant of Springer's theorem proven in [PR] there is a section
$s_U: U \to \ _UQ$
of the projection
$_UQ \to U$.
(If char(k)=2 the proof a variant of Springer's theorem
given in [PR] works well with a very mild modification).
The Theorem is proven.

\end{proof}
The Main Theorem has the following corollaries
\begin{cor}[Main1]
Let $\mathcal O_{X,\{ x_1, x_2,\dots, x_s \}}$ be the semi-local ring as above and let $k(X)$ be the rational function field on $X$.
Let $P$ be a free $\mathcal O_{X,\{ x_1, x_2,\dots, x_s \}}$-module of rank $n > 1$ and $q: P \to \mathcal O_{X,\{ x_1, x_2,\dots, x_s \}}$
be a form over $\mathcal O_{X,\{ x_1, x_2,\dots, x_s \}}$ as above, that is the $\mathcal O_{X,\{ x_1, x_2,\dots, x_s \}}$-scheme
$Q:=\{q=0\}\subset \Pro^{n-1}_{\mathcal O_{X,x}}$ is smooth over $\mathcal O_{X,x}$.
If the equation $q=0$ has a non-trivial solution over $k(X)$, then it has a unimodular solution over $\mathcal O_{X,\{ x_1, x_2,\dots, x_s \}}$.
\end{cor}

\begin{cor}[Main2]\label{Main2}
Let $R$ be a semi-local regular domain containing a field and $R$ is such that all the residue fields are infinite.
Let $K$ be the fraction field of $R$.
Let $P$ be a free $R$-module of rank $n > 1$ and $q: P \to R$
be a quadratic form over $R$ such that the $R$-scheme
$Q:=\{q=0\}\subset \Pro^{n-1}_{R}$ is smooth over $R$.
If the equation $q=0$ has a non-trivial solution over $K$,
then it has a unimodular solution over $R$.
\end{cor}

\begin{cor}[Main3]
Let $R$ be a semi-local regular domain containing a field and $R$ is such that all the residue fields are infinite.
Let $K$ be the fraction field of $R$.
Let $P$ be a free $R$-module of even rank $n > 0$ and $q: P \to R$
be a quadratic form over $R$ such that the $R$-scheme
$Q:=\{q=0\}\subset \Pro^{n-1}_{R}$ is smooth over $R$.
Let $u \in R^{\times}$ be a unit.
If $u$ is represented by $q$ over $K$, then $u$ is represented by $q$ already over $R$.

If $1/2 \in R$, then the same holds for a quadratic space of an arbitrary rank.
\end{cor}

\begin{proof}[Proof of Proposition \ref{Main}]
The following Lemma is a corollary from Lemma 2.2.1 and
Proposition 3.1.7. from \cite{Kn}
\begin{lem}
For $n>1$
there exists an affine open subset $X^0$ containing $x$ and a Galois \'{e}tale cover
$\tilde X^0 \xra{\pi} X^0$
such that
the $k[X^0]$-module
$P \otimes_{k[X]} k[X^0]$ coincides with $k[X^0]^n$ and
$\pi^*(q)$ is proportional to the quadratic space
$\bot^m_{i=1}T_iT_{i+m}$ in the case $n=2m$ and
is proportional to the semi-regular quadratic module
$\bot^m_{i=1}T_iT_{i+m}\ \bot \ T^2_n$
in the case $n=2m+1$.
\end{lem}

By this Lemma we may and will assume that $P=k[X]^n$ and that we are given with
a Galois \'{e}tale cover
$\pi: \tilde X \xra{\pi} X$
such that the quadratic space
$\pi^*(q)$ is proportional to a split quadratic space.
Let $\Gamma$ be the Galois group of $\tilde X$ over $X$.
Let $\tilde U= \pi^{-1}(U) \subset \tilde X$.

Let
$\overline {U \times X}:= (\tilde U \times \tilde X)/\Delta(\Gamma)$.
Clearly,
$U \times X = (\tilde U \times \tilde X)/(\Gamma \times \Gamma)$.
Let
$\rho: \overline {U \times X} \to U \times X$
be the obvious map.

Let
$p_2: U \times X \to X$
be projection to $X$ and
$p_1: U \times X \to U$
be the projection to $U$.
The quadratic spaces
$p^*_1(q)$
and
$p^*_2(q)$
over $U \times X$ are not proportional in general. However the following Proposition holds
(see Appendix, Lemma \ref{equating2})
\begin{prop}
\label{equating}
The quadratic spaces $\rho^*(p^*_2(q))$ and $\rho^*(p^*_1(q))$ are proportional.

\end{prop}

Further by
\cite[Prop. 3.3, Prop. 3.4]{PSV}
and
\cite{PaSV}
we may find an open $X^{\prime}$ in $X$ containing $x$ and an
open affine $S \subset \Pro^{d-1}$ (d=dim(X)) and a smooth morphism
$f^{\prime}: X^{\prime} \to S$ making $X^{\prime}$ into a smooth relative curve
over $S$ with the geometrically irreducible fibres.
Moreover we may find $f^{\prime}$ such that
$f^{\prime}|_{X^{\prime} \cap Z}: Z^{\prime}=X^{\prime} \cap Z \to S$
is finite, where $Z$ is the vanishing locus of $g \in k[X]$.
Moreover $f^{\prime}$ can be written as
$pr_S \circ \Pi^{\prime}= f^{\prime}$,
where
$\Pi^{\prime}: X^{\prime} \to \Aff^1 \times S$
is a finite surjective morphism.
Set
$\tilde X^{\prime}=\pi^{-1}(X^{\prime})$.
Replacing notation we write $X$ for $X^{\prime}$,
$\tilde X$ for $\tilde X^{\prime}$, $Z$ for $Z^{\prime}$,
$f: X \to S$ for
$f^{\prime}: X^{\prime} \to S$,
$\Pi: X \to \Aff^1 \times S$
for
$\Pi^{\prime}: X^{\prime} \to \Aff^1 \times S$.

Let
$\overline {U \times_S X}:= (\tilde U \times_S \tilde X)/\Delta(\Gamma)$.
Clearly,
$U \times_S X = (\tilde U \times_S \tilde X)/(\Gamma \times \Gamma)$.
Let
$$\rho_S: \overline {U \times_S X} \to U \times_S X$$
be the obvious map.

Let
$p_X: U \times_S X \to X$
be projection to $X$ and
$p_U: U \times_S X \to U$
be the projection to $U$.
By Proposition \ref{equating}
the quadratic spaces
$\rho^*_S(p^*_X(q))$
and
$\rho^*_S(p^*_U(q))$
are proportional.

Now the pull-back of $\Pi$ be means of the morphism
$U \hra X \to S$
defines a finite surjective morphism
$\Theta: U \times_S X \to \Aff^1 \times U$.
So,
$\Theta \circ \rho_S: \overline {U \times_S X} \to \Aff^1 \times U$
is a finite surjective morphism of $U$-schemes.
The $U$-scheme
$\overline {U \times_S X}$
is smooth over $U$ since
$U \times_S X$ is smooth over $U$
and
$\rho_S$ is \'{e}tale.
The subscheme
$\Delta(\tilde U)/\Delta(\Gamma) \subset \overline {U \times_S X}$
projects isomorphically onto $U$. So, we are given
with a section $\tilde \Delta$ of the morphism
$$\overline {U \times_S X} \xra{\rho_S} U \times_S X \xra{p_U} U.$$

The recollection from the latter paragraph shows that we are under
the hypotheses of Lemma \ref{OjPan} from Appendix B for the
relative $U$-curve $\mathcal X:= \overline {U \times_S X}$ and its
closed subset $\mathcal Z:= \rho^{-1}_S(U \times_S Z)$. (If to be
more accurate, then one should take the connected component
$\mathcal X^{c}$ of $\mathcal X$ containing $\tilde \Delta(U)$ and
the closed subset $\mathcal Z \cap \mathcal X^{c}$ of $\mathcal
X^{c}$ ).

By Lemma \ref{OjPan}
there exists an open subscheme
$\mathcal X^{0} \hra \mathcal X$
and a finite surjective morphism
$\alpha: \mathcal X^{0} \to \Aff^1 \times U$
such that $\alpha$ is \'{e}tale over
$0 \times U$ and $1 \times U$ and
$\alpha^{-1}(0 \times U) = \tilde \Delta(U)\coprod D_0$.
Moreover if we define $D_1$ as $\alpha^{-1}(1 \times U)$,
then
$D_1 \cap Z = \emptyset$
and
$D_0 \cap Z = \emptyset$.
One has
$[D_1: U]=[D_0:U]+ 1$.
Thus either
$[D_1: U]$
is odd or
$[D_0: U]$
is odd.

Assume $[D_1: U]$ is odd.
Then the morphism
$1 \times U \xleftarrow{\alpha|_{D_1}} D_1$,
the morphism
$D_1 \xra{p_X \circ \rho_S} X-Z$
and the isomorphism
$\bar \Phi:= \Phi|_{D_1}$
satisfy the conclusion of the Proposition \ref{Main}
(here $\Phi$ is from the Proposition \ref{equating}).
The Proposition is proven.

\end{proof}

\section{Appendix A: Equating Lemma}
Let $k$ be a field, $X$ be a $k$-smooth affine scheme, $G$ be a reductive $k$-group,
$\mathcal G/X$ be a principle $G$-bundle over $X$. Let
$\pi: \tilde X \to X$
be a finite \'{e}tale Galois cover of $X$ with a Galois group $\Gamma$ and let
$s: \tilde X \to \mathcal G$
be an $X$-scheme morphism ({\it in other words $\mathcal G$ splits over} $\tilde X$).
Let
$\overline {X \times X}:= (\tilde X \times \tilde X)/\Delta(\Gamma)$.
Clearly,
$X \times X = (\tilde X \times \tilde X)/(\Gamma \times \Gamma)$.
Let
$\pi: \overline {X \times X} \to X \times X$
be the obvious map. Observe that the map $\tilde X \times \tilde X \to \overline {X \times X}$
is an \'{e}tale Galois cover with the Galois group $\Gamma$.

Let $q_i: \tilde X \times \tilde X \to \tilde X$ be projection to the i-th factor
and let
$p_i: X \times X \to X$
be projection to the i-th factor.
The principal $G$ bundles
$\mathcal G_1:= p^*_1(\mathcal G)$
and
$\mathcal G_2:= p^*_2(\mathcal G)$
over $X \times X$ are not isomorphic in general. However the following Proposition holds
\begin{lem}
\label{equating2}
The principal $G$-bundles $\pi^*(\mathcal G_1)$ and $\pi^*(\mathcal G_2)$ are isomorphic
and moreover there is such an isomorphism
$\Phi: \pi^*(\mathcal G_1) \to \pi^*(\mathcal G_2)$
that the restriction of
$\Phi$ to the subscheme
$X= \Delta(\tilde X)/(\Gamma) \subset \overline {X \times X}$
is the identity isomorphism.
\end{lem}

\begin{proof}
The morphism $s: \tilde X \to \mathcal G$ gives rise to a 1-cocycle
$a: \Gamma \to G(\tilde X)$ defined as follows:
given $\gamma \in \Gamma$ consider the composition
$s \circ \gamma$ and set
$a_{\gamma} \in G(\tilde X)$
to be a unique element with
$a_{\gamma}\cdot s= s \circ \gamma$
in
$G(\tilde X)$.

It's straight forward to check that the 1-cocycle corresponding
to the principal $G$ bundle
$\pi^*(\mathcal G_1)$
and the morphism
$\tilde X \times \tilde X \xra{q_1} \tilde X \xra{s} \mathcal G$
coincides with the one
$$\Gamma \xra{a} G(\tilde X) \xra{q^*_1} G(\tilde X \times \tilde X).$$
Similarly
the 1-cocycle corresponding
to the principal $G$ bundle
$\pi^*(\mathcal G_2)$
and the morphism
$\tilde X \times \tilde X \xra{q_2} \tilde X \xra{s} \mathcal G$
coincides with the one
$$\Gamma \xra{a} G(\tilde X) \xra{q^*_2} G(\tilde X \times \tilde X).$$
Let
$b \in G(\tilde X \times \tilde X)$
be an element defined by the equality
$b\cdot (s \circ q_2) = s \circ q_1$.
To prove that
the principal $G$ bundles
$\pi^*(\mathcal G_1)$
and
$\pi^*(\mathcal G_2)$
are isomorphic it suffices to check that for every
$\gamma \in \Gamma$
the following relation holds in
$G(\tilde X \times \tilde X)$
\begin{equation}
\label{cob}
^{\gamma}b\cdot q^*_2(a)(\gamma)\cdot b^{-1}= q^*_1(a)(\gamma),
\end{equation}
where
$q^*_i(a)(\gamma):= q^*_i \circ a$
for
$i=1,2$.

To prove the relation (\ref{cob}) it suffices to check the following one in
$\mathcal G(\tilde X \times \tilde X)$
\begin{equation}
\label{cob1}
(^{\gamma}b\cdot q^*_2(a)(\gamma)\cdot b^{-1})\cdot(s \circ q_1)= q^*_1(a)(\gamma)\cdot (s \circ q_1).
\end{equation}
One has the following chain of relations
$$(^{\gamma}b\cdot q^*_2(a)(\gamma)\cdot b^{-1})\cdot (s \circ q_1)=
(\ ^{\gamma}b\cdot q^*_2(a)(\gamma))\cdot (s \circ q_2)=
\ ^{\gamma}b\cdot \ ^\gamma(s \circ q_2)=
$$
$$= \ ^{\gamma}(s \circ q_1)= s \circ q_1 \circ (\gamma \times \gamma)$$
The first one follows from the definition of the element $b$, the second one follows from the commutativity of the diagram
$$
\begin{CD}
G \times \mathcal G                    @>{\nu}>>     \mathcal G      \\
@A{(a(\gamma),s)}AA                                         @AA{s}A    \\
\tilde X                    @>{\gamma}>>            \tilde X \\
@A{q_2}AA                                         @AA{q_2}A \\
\tilde X \times \tilde X  @>{\gamma \times \gamma}>>    \tilde X \times \tilde X,
\end{CD}
$$
the third one follows from the commutativity of the diagram
$$
\begin{CD}
G \times \mathcal G                    @>{\nu}>>     \mathcal G      \\
@A{(b,s \circ q_2)}AA                                         @AA{s}A    \\
\tilde X  \times \tilde X                   @>{q_1}>>            \tilde X \\
@A{\gamma \times \gamma}AA                                          \\ 
\tilde X \times \tilde X. 
\end{CD}
$$
Thus
$(^{\gamma}b\cdot q^*_2(a)(\gamma)\cdot b^{-1})\cdot (s \circ q_1)= s \circ q_1 \circ (\gamma \times \gamma)$.
The right hand side of the relation (\ref{cob1}) is equal to
$s \circ q_1 \circ (\gamma \times \gamma)$ as well, as follows from the commutativity of the diagram
$$
\begin{CD}
G \times \mathcal G                    @>{\nu}>>     \mathcal G      \\
@A{(a(\gamma),s)}AA                                         @AA{s}A    \\
\tilde X                    @>{\gamma}>>            \tilde X \\
@A{q_1}AA                                         @AA{q_1}A \\
\tilde X \times \tilde X  @>{\gamma \times \gamma}>>    \tilde X \times \tilde X.
\end{CD}
$$
So, the relation (\ref{cob1}) holds.
Whence the relation (\ref{cob}) holds.
Whence the principal $G$ bundles
$\pi^*(\mathcal G_1)$
and
$\pi^*(\mathcal G_2)$
are isomorphic.

The composite
$\tilde X \xra{\Delta} \tilde X \times \tilde X \xra{q_2} \tilde X \xra{s} \mathcal G$
equals $s$ and equals the composite
$s \circ q_1 \circ \Delta$.
Whence
$\Delta^*(b)= 1 \in G(\tilde X)$.
This shows that the restriction to $X= \Delta(X)/\Delta(\Gamma)$ of the isomorphism
$\pi^*(\mathcal G_1)$ and $\pi^*(\mathcal G_2)$
corresponding to the element $b$
is the identity isomorphism.
The Lemma is proved.

\end{proof}

\section{Appendix B: a variant of geometric lemma}
Let $k$ be an infinite field, $Y$ be a $k$-smooth algebraic variety, $y \in Y$ be a point,
$\mathcal O= \mathcal O_{Y,y}$ be the local ring, $U=\text{Spec}(\mathcal O)$. Let
$\mathcal X/U$ be a $U$-smooth relative curve with geometrically connected fibres equipped with
a finite surjective morphism
$\pi: \mathcal X \to \Aff^1 \times U$
and equipped with a section
$\Delta: U \to \mathcal X$ of the projection $p: \mathcal X \to U$.
Let
$\mathcal Z \subset \mathcal X$
be a closed subset finite over $U$.
The following Lemma is a variant of Lemma 5.1 from
\cite{OP}.

\begin{lem}
\label{OjPan}
There exists an open subscheme
$\mathcal X^{0} \hra \mathcal X$
and a finite surjective morphism
$\alpha: \mathcal X^{0} \to \Aff^1 \times U$
such that $\alpha$ is \'{e}tale over
$0 \times U$ and $1 \times U$ and
$\alpha^{-1}(0 \times U) = \Delta(U)\coprod D_0$.
Moreover if we define $D_1$ as $\alpha^{-1}(1 \times U)$,
then
$D_1 \cap \mathcal Z = \emptyset$
and
$D_0 \cap \mathcal Z = \emptyset$.
\end{lem}

\begin{proof}
Let $\bar {\mathcal X}$ be the normalization of the scheme $\Pro^1 \times U$ in the function field $k(\mathcal X)$ of $\mathcal X$.
Let $\bar \pi: \bar {\mathcal X} \to \Pro^1 \times U$ be the morphism.
Let $\mathcal X_{\infty}=\pi^{-1}(\infty \times U)$ be the set theoretic preimage of $\infty \times U$.
Let
$\bar p: \bar {\mathcal X} \to U$ be the structure map.
Let
$u \in U$
be the closed point and
$\bar X_u= \bar {\mathcal X} \times_U u$.

Let , $L^{\prime}=\bar \pi^{*}(\mathcal O_{\Pro^1 \times U}(1))$,
$L^{\prime\prime}=\mathcal O_{\bar X}(\Delta(U))$. Let
$D_{\infty}= (\pi^{*})(\infty \times U)$ be the pull-back of the
Cartier divisor $\infty \times U \subset \Pro^1 \times U$. Choose
and fix a closed embedding $i: \bar {\mathcal X} \hra \Pro^{n}
\times U$ of $U$-schemes. Set $L= i^{*}(\mathcal O_{\Pro^n \times
U}(1))$.

The sheaf $L$ is very ample. Thus the sheaf
$L^{\prime\prime} \otimes L$
is very amply as well. So, there exists a closed embedding
$i^{\prime\prime}: \bar {\mathcal X} \hra \Pro^{n^{\prime\prime}} \times U$
of $U$-schemes such that
$L^{\prime\prime} \otimes L=(i^{\prime\prime})^*(\mathcal O_{\Pro^{n^{\prime\prime}} \times U}(1)).$
Using Bertini theorem choose a hyperplane
$H^{\prime\prime} \subset \Pro^{n^{\prime\prime}} \times U$
such that \\
$(a^{\prime\prime})$ $H^{\prime\prime}\cap \Delta(U) =\emptyset$, $H^{\prime\prime}\cap \mathcal Z =\emptyset$,
$H^{\prime\prime}\cap D_{\infty} =\emptyset$.\\
Define a Cartier divisor $D^{\prime\prime}$ on $\bar {\mathcal X}$
as the the closed subscheme
$H^{\prime\prime} \cap {\bar {\mathcal X}}$ of ${\bar {\mathcal X}}$.\\
Regard $D^{\prime\prime}_1:= D^{\prime\prime} \coprod D_{\infty}$
as a Cartier divisor on $\bar {\mathcal X}$. Clearly, one has
$\mathcal O_{\bar {\mathcal X}}(D^{\prime\prime}_1)=
L^{\prime\prime} \otimes L \otimes L^{\prime}$.

The sheaf $L$ is very ample. Thus the sheaf
$L^{\prime} \otimes L$
is very ample as well. So, there exists a closed embedding
$i^{\prime}: \bar {\mathcal X} \hra \Pro^{n^{\prime}} \times U$
of $U$-schemes such that
$L^{\prime} \otimes L=(i^{\prime})^*(\mathcal O_{\Pro^{n^{\prime}} \times U}(1)).$
Using Bertini theorem choose a hyperplane
$H^{\prime} \subset \Pro^{n^{\prime}} \times U$
such that \\
$(a^{\prime})$
$H^{\prime}\cap \Delta(U) =\emptyset$, $H^{\prime}\cap \mathcal Z =\emptyset$,
$H^{\prime}\cap D^{\prime\prime}_1 =\emptyset$;\\
$(b^{\prime})$
the scheme theoretic intersection
$H^{\prime}\cap \bar X_u$ is a $k(u)$-smooth scheme.\\
Define a Cartier divisor $D^{\prime}$ on $\bar {\mathcal X}$ as
the closed subscheme
$D^{\prime}=H^{\prime} \cap \bar {\mathcal X}$ of ${\bar {\mathcal X}}$.\\

Regard $D^{\prime}_1:= D^{\prime} \coprod \Delta(U)$ as a Cartier
divisor on $\bar {\mathcal X}$. Clearly, one has $\mathcal O_{\bar
{\mathcal X}}(D^{\prime}_1)= L^{\prime} \otimes L \otimes
L^{\prime\prime}$.

Observe that $D^{\prime}$ is an essentially $k$-smooth scheme
finite and \'{e}tale over $U$. Let $s^{\prime}$ and
$s^{\prime\prime}$ be global sections of $L^{\prime} \otimes L
\otimes L^{\prime\prime}$ such that the vanishing locus of
$s^{\prime}$ is the Cartier divisor $D^{\prime}_1$ and the
vanishing locus of $s^{\prime\prime}$ is the Cartier divisor
$D^{\prime\prime}_1$. Clearly $D^{\prime}_1 \cap
D^{\prime\prime}_1=\emptyset$. Thus $f=[s^{\prime}:
s^{\prime\prime}]: \bar {\mathcal X} \to \Pro^1$ is a regular
morphism of $U$-schemes. Set
$$\bar \alpha= (f,\bar p): \bar {\mathcal X} \to \Pro^1 \times U.$$

Clearly,
$\bar \alpha$
is a finite surjective morphism. Set
$\mathcal X^0= \bar \alpha^{-1}(\Aff^1 \times U)$
and
$$\alpha=\bar \alpha|_{\mathcal X^0}: \mathcal X^0 \to \Aff^1 \times U.$$
Clearly, $\alpha$ is a finite surjective morphism and $\mathcal X^0$ is an open subscheme of $\mathcal X$.
Since $\alpha$ is a finite surjective morphism and $\mathcal X^0$, $\Aff^1 \times U$ are regular schemes
the morphism $\alpha$ is flat by a theorem of Grothendieck. Since
$D^{\prime}_1$
is finite \'{e}tale over $U$
the morphism $\alpha$ is \'{e}tale over $0 \times U$.
So, we may choose a point $1 \in \Pro^1$ such that
the $\alpha$ is \'{e}tale over $1 \times U$
and
$(\alpha)^{-1}(1 \times U) \cap \mathcal Z=\emptyset$.
If we set
$D_0=D^{\prime}_1$, then
$\alpha^{-1}(0 \times U) = \Delta(U)\coprod D_0$
and
$D_0 \cap \mathcal Z=\emptyset$.
The Lemma is proven.

\end{proof}


\end{document}